\documentclass[review]{elsarticle}

\usepackage[margin=1in]{geometry}  
\usepackage{graphicx}              
\usepackage{amsmath}               
\usepackage{amsfonts}              
\usepackage{amsthm}                
\usepackage{algorithm}
\usepackage{algorithmic}
\usepackage{varwidth}
\usepackage{parskip}
\usepackage{rotating}
\usepackage{pdflscape}
\usepackage[numbers]{natbib}
\usepackage{caption}
\usepackage{subcaption}

\newcommand{\bea}{\begin{eqnarray}}
\newcommand{\eea}{\end{eqnarray}}
\newcommand{\bee}{\begin{eqnarray*}}
\newcommand{\eee}{\end{eqnarray*}}
\newcommand{\nn}{\nonumber}
\newcommand{\lb}{\label}
\newcommand{\la}{\lambda}
\newtheorem{thm}{Theorem}[section]
\newtheorem{lem}[thm]{Lemma}

\newtheorem{remark}{Remark}

\begin{document}

\begin{frontmatter}

\title{Modeling and analysis of a discrete-time group-arrival and batch-size-dependent service queue with single and multiple vacation}

\author{S. Pradhan\corref{mycorrespondingauthor}}
\ead[url]{spiitkgp11@gmail.com}
\cortext[mycorrespondingauthor]{Corresponding author}
\ead{spiitkgp11@gmail.com}
\author{N. Nandy}
\ead{nilanjannanday2@gmail.com}
\address{Department of Mathematics, Visvesvaraya National Institute of Technology, Nagpur-440010, India}

\begin{abstract}
Discrete-time queueing system has widespread applications in packet switching networks, internet protocol, Broadband
Integrated Services Digital Network (B-ISDN), circuit switched time-division multiple access etc. In this paper, we analyze an infinite-buffer discrete-time group-arrival queue with single and multiple vacation policy where processing/transmission/service times dynamically vary with batch-size under service. The bivariate probability generating function of queue length and server content distribution has been derived first, and from that we extract the probabilities
in terms of roots of the characteristic equation. We also find the joint distribution at arbitrary and pre-arrival slot. Discrete phase type distribution plays
a noteworthy role in order to regulate the high transmission error through a particular channel. In view of this, numerical illustrations include the service
time distribution to follow the discrete phase type distribution.

\end{abstract}

\end{frontmatter}
\section{Introduction}
Since last few decades, discrete-time queueing models have been a major concern of research due to its wide range of applications in data transmission over satellite, cable access networks, ATM switching elements, IEEE802.11 n WLANs etc. The fundamental mechanism is based on the packet switching principle in which data processing takes place in regular and equally spaced time intervals defined as slots. Detailed applications can be found in Bruneel and Kim \cite{bruneel1993discrete}, Alfa \cite{alfa2010queueing}, Samanta et al. \cite{samanta2007discrete,samanta2007analyzing}, Gupta et al. \cite{gupta2007discrete,gupta2014analysis}, Claeys et al. \cite{claeys2010complete,claeys2010queueing} and references therein.\\
\hspace*{0.3cm}In many real-life situations, it is observed that after the service completion of a batch, the number of customers in the queue is less than the pre-defined minimum threshold, and hence service can not be provided in the next round until the minimum threshold has been accumulated. During this time, the server may leave the system for an arbitrary amount of time to do some other work. These type of queues are known as vacation queues and has potential applications in processor schedules in computer and switching system, shared resource maintenance, designing of local area networks (LAN), manufacturing system with server breakdown etc. Various type of vacation queueing models for both discrete- and continuous- time set-up can be found in Dhosi \cite{doshi1986queueing}, Baba\cite{baba1986mx}, Takagi \cite{takagi1993queueing}, Lee et al.\cite{lee1994analysis}, Tian \cite{tian2006vacation}, Reddy et al.\cite{reddy1999non},  Gupta et al.\cite{gupta2004finite,gupta2005finite}, Sikdar et al.\cite{sikdar2005queue}, Jeyakumar et al.\cite{jeyakumar2012study} etc.\\
\hspace*{0.3cm}Batch-service queues have applications in mobile crowdsourcing app for smart cities, recreational devices in amusement park, group testing of blood samples for detecting HIV/Influenza/viruses or to eliminate defective items in manufacturing systems etc. These can also be applied in packet switched telecommunication networks by regulating the transmission rate for developing congestion control techniques.  Recent development of batch-service queues can be found in Goswami et al. \cite{goswami2006analyzing,goswami2006discrete}, Germs and Foreest \cite{germs2013analysis}, Barbhuiya and Gupta \cite{barbhuiya2019discrete}, Bank and Samanta \cite{bank2020analytical} and references therein.\\
\hspace*{0.3cm}In recent times, a few researchers have focused on batch-service queues with batch-size-dependent service due to their usefulness in production and transportation, package delivery, group testing of blood/urine samples, etc. Batch-size-dependent service policy also helps us to reduce the congestions. For the betterment of human civilization, a tremendous research is rapidly growing on this type of queues. By considering discrete-time $Geo^X/G_n^{(l,c)}/1$ and D-BMAP$/G^{(l,c)}_r/1$ queue, Claeys et al.  \cite{claeys2011analysis,claeys2013analysis} derived joint probability/vector generating function (pgf/vgf) of queue and server content at arbitrary slot. Claeys et al. \cite{claeys2013tail} also focused on the analysis of tail probabilities for the customer delay. They also focused on the influence of correlation of the arrival process to study the behavior of the system. However, the complete extraction procedure of the joint probabilities was not discussed.  Banerjee et al. \cite{banerjee2014analysis} provided the complete joint distribution of queue and server content for a discrete-time finite-buffer $Geo/G_r^{(a,b)}/1/N$ queue. Yu and Alfa \cite{yu2015algorithm} analyzed D-MAP$/G_r^{(1,a,b)}/1/N$ queue in which they reported the queue length and server content distribution together using both embedded Markov chain technique (EMCT) and quasi birth and death (QBD) process. Unfortunately, in these analysis no one included any type of vacation policy. On the contrary, for the continuous-time queue, Banerjee et al.\cite{banerjee2012reducing,banerjee2013analysis}, Pradhan et al.\cite{pradhan2017analysis,pradhan2020distribution,pradhan2019analysis} analyzed batch-size-dependent service queue and provided the detailed extraction procedure. \\
\hspace*{0.3cm}Recently, Nandy and Pradhan \cite{onthejoint2020} analyzed an infinite-buffer discrete-time batch-size-independent service $Geo/G^{(a,b)}_n/1$ queue with single and multiple vacation where customers arrive individually according to Bernoulli arrival process. However, in several practical scenarios, customers may arrive in groups or batches. In order to cope with those circumstances, it is necessary to analyze $Geo^X/G^{(a,b)}_n/1$ queue as the inclusion of group/batch arrival makes the analysis more complex and challenging. The governing equations, probability generating functions, joint probabilities at different epochs, numerical illustrations are completely different from the previous $Geo/G^{(a,b)}_n/1$ queueing model. To be more precise, the joint queue and server content distribution for the discrete-time, group arrival batch-size-independent service $Geo^X/G^{(a,b)}_n/1$ queue with single and multiple vacation is not available so far in the literature to the best of authors' knowledge. The server capability can be maximized to reduce average waiting time of a customer as well as cost of the system for which the server content distribution is an essential tool. It also has to be kept in mind that the server utilization and processing can be met together for maximum response for which system designer is significantly concerned about both queue and server content distribution together.\\
\hspace*{0.3cm}On account of all those perspective, in this paper, we focus on both the queue and server content distribution for an infinite-buffer group-arrival discrete-time batch processing transmission channel with batch-size-dependent service policy, and single and multiple vacation. It should be noted here that group-arrival along with single and multiple vacation as well as batch-size-dependent service policy makes the mathematical analysis quite complicated from both analytical and computational point of view. Firstly, we develop the steady-state governing equations of the concerned queueing model. Secondly, we focus on the derivation of the bivariate pgf of queue and server content distribution at service completion epoch which is the central focal part of our analysis. After the extraction of joint probabilities from the bivariate pgf we generate a relationship with the arbitrary slot probabilities. Some relevant marginal distributions and performance measures are provided which fulfills the essence of the system designer. Finally the results are illustrated with suitable and interesting numerical examples which is fruitful to real-life circumstances.\\
\hspace*{0.3cm}Rest of the paper is organized as follows: Section \ref{PA} describes the suggested application of the queueing model while \ref{MD} gives the detailed model description. The steady-state equations of the model is described in Section \ref{SE}. Procedure of finding joint probability distribution at service completion epoch is described in Section \ref{JDD} whereas Section \ref{JDA} gives the relation between joint probabilities at arbitrary slot and service completion epoch. Marginal distributions and performance measures are provided in Section \ref{MaD}. Numerical examples are sketched in Section \ref{NR} followed by the conclusion.
\section{Practical application of the suggested queueing system}\label{PA}
The described system has several real-life applications. In modern telecommunication system the messages, data, video, images, signals are broken into manageable information packets initially. Since the arrival occurs in batches the situation becomes more realistic as the data in the form of group of packets can arrive from various places at a time. The packets are transmitted in a fixed slot as a single entity through a common interface in ATM multiplexing and switching technology, IEEE802.11 and WLANs, internet protocols with a minimum threshold and maximum limit. In this context, one can consider the service/transmission rule as general bulk service $\left(a, b\right)$ rule. The mechanism behind this is to construct only one header per aggregate batch instead of one header per single entity. As a consequence, the system become more efficient and improves the quality of service by reducing both processing time and total cost.\\
\hspace*{0.3cm}In communication system several processors are used for which considerable amount of testing and maintenance time is required. When the number of customers are less than the minimum threshold in the queue the maintenance work can be completed for a random amount of time which can be thought as vacation time. After the maintenance work it provides high reliability of the system and using that the congestion and lags in later services may be reduced. Excessive delays of transmission time can be bypassed using batch-size-dependent service mechanism as the processing/transmission rate dynamically vary with the size of packets.
\section{\textbf{Model description}}\label{MD}
\begin{itemize}

\item \emph{Discrete-time set up:}  The time axis is slotted into equal lengths. Let the time axis be marked by $0$, $\omega$, $2 \omega$,\ldots, $k \omega$,\ldots. For our convenience we set $\omega = 1$, so that each interval is of unit length. The time axis is now $0$, $1$, $2$,\ldots, $k$, \ldots Here we discuss the model for late arrival delayed access system (LAS-DA) and thus a potential batch arrival occurs in $\left(k-, k\right)$ and the interval of batch-departure is $\left(k,k+\right)$. For the discrete-time set-up, an arrival and a departure/transmission takes place simultaneously at a slot boundary. Different epochs of LAS-DA system are shown in figure\ref{figg1}.
\begin{figure}[h!]
	\begin{center}
\includegraphics[scale=0.6]{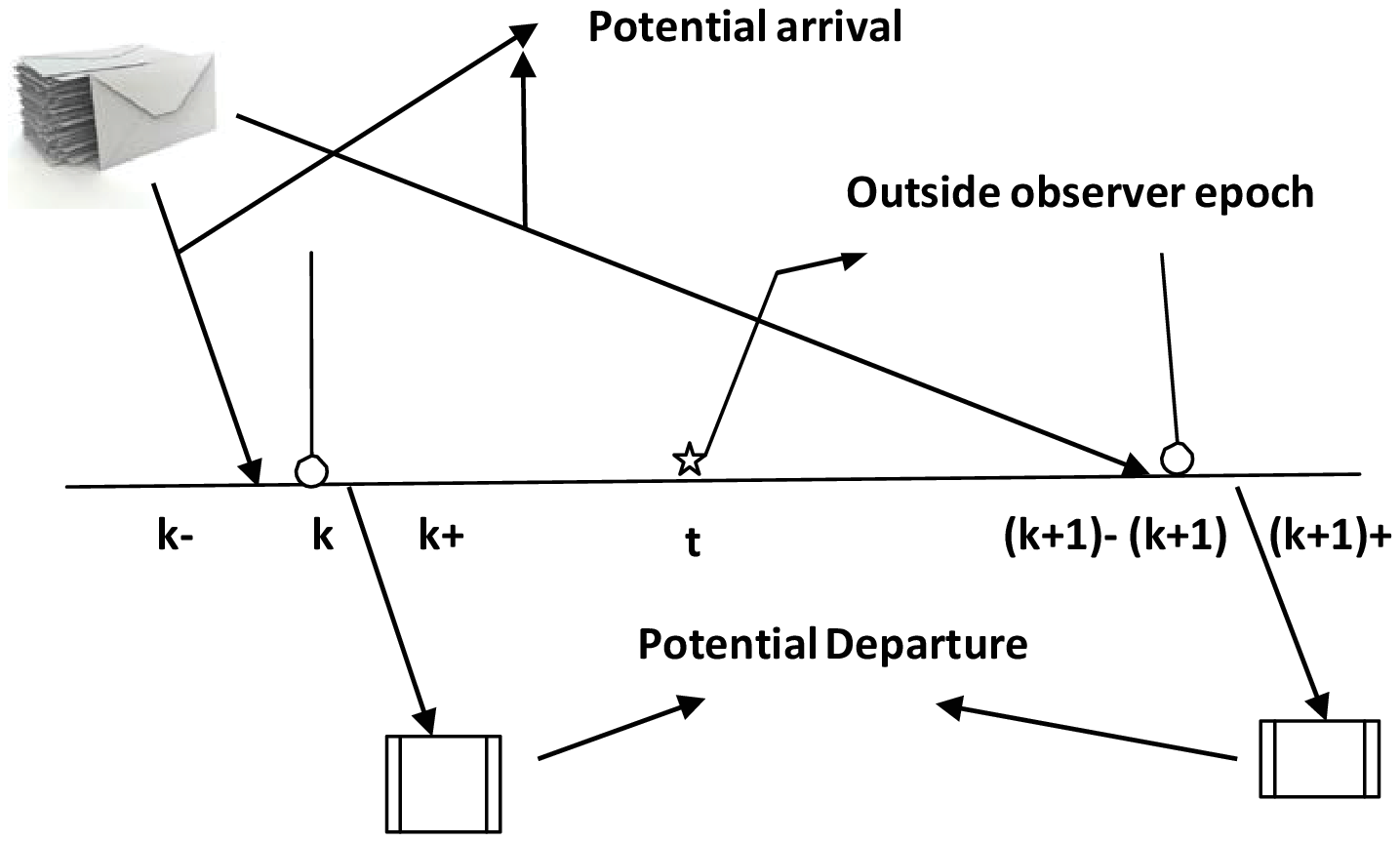}
	\end{center}
	\caption{Different epochs of LAS-DA system}\lb{figg1}
\end{figure}

\item \emph{Arrival process:} Here the group of customers/messages in terms of packets arrive according to compound bernoulli process with mean batch arrival rate $\lambda$, which is analytically tractable and fractal characteristics like self-similarity is absent. Thus the inter-arrival time follows geometric distribution with probability mass function (pmf) $\psi_{n} = {\bar{\lambda}}^{n-1}\lambda$, $0 \leq \lambda \leq 1$, $n \geq 1$ and $\bar{\lambda}=1-\lambda$. The size of arriving group of customers/messages in a slot are identically distributed random variables with probability distribution $Pr(X=m)= g_{m}$, $m$ $\epsilon$ $\mathbb{N}$, where $X$ is the generic batch size with finite mean $E(X)= \bar{g}$ and associated pgf $G(z)= \sum_{m=1}^{\infty} g_{m}z^{m}$.

\item \emph{Batch-service rule:} The packets are transmitted in group/batches according to general bulk service $(a,b)$ rule. The server initiates the service only when there is minimum `$a$' number of customers/packets present in the queue for service. For the queue size $r$ $(a \leq r \leq b)$, entire group is taken into service. If the size of the queue exceeds `$b$', then the server transmits exactly `$b$' packets and rest of them wait for next round of service. A newly arrived customer/message can not join the ongoing service even if there is a free capacity.
	
\item \emph{Service process:} In most of the real-life scenarios, the transmission/processing time does not follow any particular well known probability distribution. In order to tackle this we assume that the service time follows a general distribution using that a large class of probability distribution can be covered. We also assume that the service time is dependent on the batch size under service. This service mechanism reduces congestion in a telecommunication networks and improves the system productivity. Let the random variable $T_{i}$, $a \leq i \leq b$, be defined as the service time of a batch of $i$ customers/messages with pmf $Pr(T_{i}=n) = s_{i}(n)$, $n=1,2,3,\ldots$ and the corresponding pgf is given by $S^*_{i}(z) = \sum_{n=1}^{\infty}s_{i}(n)z^{n}$. Also the mean service time of batch of size $i$ is defined as $s_{i} = \frac{1}{\mu_{i}} = S_i^{*(1)}(1)$, $a \leq i \leq b$, where $S_i^{*(1)}(1)$ is the first order derivative evaluated at $z=1$.

\item \emph{Vacation process:} The concerned queueing model is analyzed with two types of vacation rules viz., single and multiple vacation with an indicator variable $\delta_{p}$ defined as follows
\begin{align*}
	\delta_{p} =\left\{\begin{array}{r@{\mskip\thickmuskip}l}
	0, ~~& \text{for single vacation,} \\
	1, ~~& \text{for multiple vacation}
	\end{array}\right.
\end{align*}
	It should be noted here that the server must decide  the pre-defined vacation rule before the service initiation. The results for the corresponding queue with single and multiple vacation can be obtained by substituting the appropriate value of the indicator function as defined above.
	
\begin{itemize}
		\item \emph{Single vacation rule:} After the completion of the service of a batch, if the server finds at least `$a$' customers waiting, it continues the service, else it leaves for a vacation of random length with pmf $v_{n}$, $n \geq 1$, pgf $V^*(z) = \sum_{n=1}^{\infty} v_{n}z^{n}$ and mean $E(V) = \bar{v} = V^{*(1)}(1)$. At the vacation termination epoch, if queue length is less than `$a$', the server waits until there are at least `$a$' customers in the queue.
		
\item \emph{Multiple vacation rule:} After the service completion of a batch if there is at least `$a$' customers in the queue it continues the service, else it leaves for a vacation of random length. After returning from the vacation if there is still not `$a$' customers have been accumulated, it leaves for another vacation of random time and the process continues unless there are `$a$' customers waiting in the queue and then the service starts. The pmf of vacation time of random length is defined as $v_{n}$, $n \geq 1$ with pgf $V^{*}(z) = \sum_{n=1}^{\infty}v_{n}z^{n}$ and mean $E(V) = \bar{v} = V^{*(1)}(1)$.
\end{itemize}
	
\item \emph{Traffic intensity}: The traffic intensity of the system is given by $\rho = \frac{\lambda \bar{g}}{b \mu_{b}} < 1$, which ensures the system stability.

\end{itemize}

\section{Governing system equations}\label{SE}
This section contains the governing equations of the model in the steady-state. To present the system equations we first define the state of the system at time $t$ as:
\begin{itemize}
\item  $N_{q}(k-) \equiv$ Number of packets/customers in the queue waiting to be transmitted,
\item $N_{s}(k-) \equiv$ Number of packets/customer with the server,
\item $U(k-) \equiv$ Remaining service time of a batch in service (if any),
\item $V(k-) \equiv$ Remaining vacation time of the server excluding the current vacation slot,
\item $\zeta(t-) \equiv$ State of the server defined as:
\begin{align*} \zeta(t-) = \left\{\begin{array}{r@{\mskip\thickmuskip}l}
0,~~& \text{when the server is in dormant state,}\\
1, ~~& \text{when the server is in vacation,}\\
2, ~~& \text{when the server is busy}
\end{array}\right.
\end{align*}
\end{itemize}
 Further, we define the following probabilities as:
\begin{eqnarray}
p_{n,0}(k-) &=& Pr\{N_{q}(k-)=n, N_{s}(k-)=0, \zeta(t-)=0\},\quad 0 \leq n \leq a-1 \nn\\
p_{n,r}(u, k-) &=& Pr\{N_{q}(k-)=n, N_{s}(k-)=r, U(k-)=u, \zeta(t-)=2\}, \quad
u \geq 1, n\geq 0, a\leq r \leq b\nn\\
Q_{n}(u,k-) &=& Pr\{N_{q}(k-)=n, V(k-)=u, \zeta(t-)=1\}, \quad u \geq 1,n \geq 0.\nn
\end{eqnarray}
As the model is analyzed in the steady state, we define the limiting probabilities as
\begin{eqnarray}
p_{n,0} &=& \lim\limits_{k- \to \infty} p_{n,0}(k-), \quad 0 \leq n \leq a-1,\nn\\
p_{n,r}(u) &=& \lim\limits_{k- \to \infty} p_{n,r}(u,k-), \quad u \geq 1,~ n \geq 0,~ a \leq r \leq b,\nn\\
Q_{n}(u) &=& \lim\limits_{k- \to \infty} Q_{n}(u,k-), \quad u \geq 1, ~n \geq 0\nn.
\end{eqnarray}
Observing the states of the system at the epochs $k-$ and $(k+1)-$ and using the supplementary variable technique (SVT), we obtain the following equations in the steady state
\begin{eqnarray}
p_{0,0} &=& \bigg[\bar{\lambda}p_{0,0} + \bar{\lambda}Q_{0}(1)\bigg]\left(1-{\delta}_{p}\right)\label{eq1}\\
p_{n,0} &=&\bigg[\bar{\lambda}p_{n,0} + \lambda\sum_{i=1}^{n}g_{i}p_{n-i,0}+ \bar{\lambda}Q_{n}(1) + \lambda\sum_{i=1}^{n}g_{i}Q_{n-i}(1)\bigg]\left(1-{\delta}_{p}\right),~~~ 1\leq n \leq a-1 \label{eq2}\\
p_{0,r}(u)&=&\bar{\lambda}p_{0,r}(u+1) + \bar{\lambda}\sum_{m=a}^{b}p_{r,m}(1)s_{r}(u) + \lambda\sum_{i=1}^{r}\sum_{m=a}^{b} g_{i} p_{r-i,m}(1)s_{r}(u) + \bar{\lambda}Q_{r}(1)s_{r}(u)\nonumber\\
&&~~~~~~~~~~~~~~~~~~ + \lambda\sum_{i=1}^{r} g_i Q_{r-i}(1)s_{r}(u)+ \left(1-{\delta}_{p}\right)\lambda\sum_{i=0}^{a-1}g_{r-i} p_{i,0}s_{r}(u),~~~ a \leq r \leq b \label{eq3}\\
p_{n,r}(u) &=& \bar{\lambda}p_{n,r}(u+1) + \lambda\sum_{i=1}^{n}g_{i} p_{n-i,r}(u+1), \quad n \geq 1,
\quad a \leq r \leq b-1\label{eq4}\\
p_{n,b}(u) &=& \bar{\lambda}p_{n,b}(u+1)+\lambda\sum_{i=1}^{n}g_{i} p_{n-i,b}(u+1) +\bar{\lambda}\sum_{m=a}^{b}p_{n+b,m}(1)s_{b}(u) + \lambda\sum_{i=1}^{n+b} \sum_{m=a}^{b} g_{i} p_{n+b-i,m}(1)s_{b}(u) \nn\\
&&+ \bar{\lambda}Q_{n+b}(1)s_{b}(u) + \lambda\sum_{i=1}^{n+b} g_{i} Q_{n+b-i}(1)s_{b}(u) + \left(1-{\delta}_{p}\right)\lambda\sum_{i=0}^{a-1}g_{n+b-i} p_{i,0}s_{b}(u),  ~~~ n\geq 1 \label{eq5}\eea
\bea
Q_{0}(u) &=& \bar{\lambda}Q_{0}(u+1) + \bar{\lambda} \left(\sum_{m=a}^{b}p_{0,m}(1) + {\delta_{p}} Q_{0}(1)\right)v(u)\label{eq6}\\
Q_{n}(u) &=& \bar{\lambda}Q_{n}(u+1) + \lambda \sum_{i=1}^{n}g_{i}Q_{n-i}(u+1)
+\bar{\lambda} \left(\sum_{m=a}^{b}p_{n,m}(1) + \delta_{p} Q_{n}(1)\right)v(u) \nonumber\\
&&~~~~~~~~~+\lambda\sum_{i=1}^{n} g_{i}\left(\sum_{m=a}^{b}p_{n-i,m}(1) + \delta_{p} Q_{n-i}(1)\right) v(u), \quad 1\leq n \leq a-1 \label{eq7}\\
Q_{n}(u) &=& \bar{\lambda}Q_{n}(u+1) + \lambda \sum_{i=1}^{n} g_{i} Q_{n-i}(u+1), \quad n \geq a \label{eq8}
\end{eqnarray}
Let us define the followings:
\begin{eqnarray}
p^*_{n,r}(z) &=& \sum_{u=1}^{\infty}p_{n,r}(u) z^{u}, \quad n \geq 0, \quad a \leq r \leq b\nonumber\\
Q^*_{n}(z) &=& \sum_{u=1}^{\infty} Q_{n}(u)z^{u}, \quad n \geq 0.\nonumber
\end{eqnarray}
Consequently, the following two results immediately follows from the above definitions. These will be used in later analysis.
\begin{eqnarray}
p_{n,r} &\equiv& p^*_{n,r}(1) = \sum_{u=1}^{\infty} p_{n,r}(u), \quad  \quad n\geq 0, a\leq r \leq b\nonumber\\
Q_{n} &\equiv& Q^*_{n}(1) = \sum_{u=1}^{\infty} Q_{n}(u), \quad n \geq 0.\nonumber
\end{eqnarray}
Since our foremost objective is to get the joint distributions from (\ref{eq1}) to (\ref{eq8}), multiplying (\ref{eq3}) to (\ref{eq8}) by $z^{u}$ and summing over $u$ from $1$ to $\infty$ we obtain
\begin{eqnarray}
\left(\frac{z-\bar{\lambda}}{z}\right)p^*_{0,r}(z) &=& \sum_{m=a}^{b}\left(\bar{\lambda}p_{r,m}+\lambda\sum_{i=1}^{r}g_{i}
p_{r-i,m}(1)\right)S^*_{r}(z) + \left(\bar{\lambda}Q_{r}(1) + \lambda\sum_{i=1}^{r}g_{i} Q_{r-i}(1)\right)S^*_{r}(z)\nn\\
&&~~~~~~~~~~~~~~~~~~~~~~~~~~~~~~+\left(1-\delta_{p}\right) \lambda\sum_{i=0}^{a-1}g_{r-i} p_{i,0} S^*_{r}(z)-\bar{\lambda}p_{0,r}(1), \quad a \leq r \leq b\label{eq9}\\
\left(\frac{z-\bar{\lambda}}{z}\right)p^*_{n,r}(z) &=&  \frac{\lambda}{z}\sum_{i=1}^{n}g_{i}p^*_{n-i,r}(z)-\bar{\lambda}p_{n,r}(1) - \lambda\sum_{i=1}^{n} g_i p_{n-i,r}(1) ,~~~n \geq 1,~~~a \leq r \leq b-1 \label{eq10}\\
\left(\frac{z-\bar{\lambda}}{z}\right)p^*_{n,b}(z) &=&  \frac{\lambda}{z}\sum_{i=1}^{n}g_{i }p^*_{n-i,b}(z) +\sum_{m=a}^{b}\left(\bar{\lambda}p_{n+b,m}(1)+\lambda
\sum_{i=1}^{n+b}g_{i}p_{n+b-i}(1)\right)S^*_{b}(z)\nn
\\&&+ \left(\bar{\lambda}Q_{n+b}(1)+\la\sum_{i=1}^{n+b}g_{i} Q_{n+b-i}(1)\right)S^*_{b}(z) +\left(1-\delta_{p}\right)\lambda\sum_{i=0}^{a-1}g_{n+b-i}p_{i,0} S^*_{b}(z)\nn\\
&&-\bar{\lambda}p_{n,b}(1) - \lambda\sum_{i=1}^{n}g_{i} p_{n-i,b}(1),\quad n \geq 1\label{eq11}\eea
\bea
\left(\frac{z-\bar{\lambda}}{z}\right)Q^*_{0}(z) &=&  \bar{\lambda} \sum_{r=a}^{b} p_{0,r}(1) V^*(z) + \delta_{p} \bar{\lambda} Q_{0}(1)V^*(z)-\bar{\lambda} Q_{0}(1)\label{eq12}\\
\left(\frac{z-\bar{\lambda}}{z}\right)Q^*_{n}(z) &=& \frac{\lambda}{z}\sum_{i=1}^{n}g_{i}Q^*_{n-i}(z)+\sum_{r=a}^{b}\biggl[\bar{\lambda}p_{n,r}(1) + \lambda\sum_{i=1}^{n}g_{i} p_{n-i,r}(1)\biggr]V^*(z) \nn \\
&&+\delta_{p}\bigg[\bar{\lambda}Q_{n}(1)+\lambda\sum_{i=1}^{n}g_{i} Q_{n-i}(1)\bigg]V^*(z)-\bar{\lambda}Q_{n}(1)-\lambda \sum_{i=1}^{n}g_{i} Q_{n-i}(1),\quad 1 \leq n \leq a-1\label{eq13}\\
\left(\frac{z-\bar{\lambda}}{z}\right)Q^*_{n}(z) &=&  \frac{\lambda}{z}\sum_{i=1}^{n}g_{i}Q^*_{n-i}(z)-\bar{\lambda}Q_{n}(1)-\lambda \sum_{i=1}^{n}g_{i}Q_{n-i}(1), \quad n \geq a \label{eq14}
\end{eqnarray}
\noindent Now using equations (\ref{eq1}), (\ref{eq2}) and (\ref{eq9})-(\ref{eq14}) we derive two results (represented as lemmas) which will be used later to obtain further results.
\begin{lem} \label{SElem1}
The probabilities $\left(p^+_{n,r},p_{n,r}(1)\right)$ and $\left(Q^+_{n},Q_{n}(1)\right)$ are connected by  the relation
	\begin{eqnarray}
	p^+_{0,r} &=& \tau^{-1}\left(\bar{\lambda} p_{0,r}(1)\right), \quad a \leq r \leq b \label{eq15}
	\\p^+_{n,r} &=& \tau^{-1}\left(\bar{\lambda}p_{n,r}(1) + \lambda\sum_{i=1}^{n} g_{i}p_{n-i,r}(1)\right), \quad n \geq 1, \quad a \leq r \leq b\label{eq16}\\
 Q^+_{0} &=& \tau^{-1}\left(\bar{\lambda}Q_{0}(1)\right)\label{eq17}\\
 Q^+_{n} &=& \tau^{-1}\left(\bar{\lambda}Q_{n}(1) + \lambda\sum_{i=1}^{n}g_{i} Q_{n-i}(1)\right) ,\quad n \geq 1\label{eq18}
	\end{eqnarray}
	where
	$\tau = \sum_{m=0}^{\infty}\sum_{r=a}^{b}p_{m,r}(1) + \sum_{m=0}^{\infty} Q_{m}(1)$.
\end{lem}
\noindent Equations (\ref{eq15}) and (\ref{eq16}) represents the relation between joint queue and server content distribution at service completion epoch with the joint distribution of queue and server content when the service is about to complete. Similar can be said about (\ref{eq17}) and (\ref{eq18}) in terms of vacation.
$\tau^{-1}$ gives us the mean of completion of service or termination of vacation per unit time, i.e, mean departure rate from the busy state or vacation state.
\begin{proof}[\textbf{Proof.}]
Here $p_{n,r}(1)$ denotes the joint probability that there are $n$ customers in the queue and $r$ with the server and remaining service time is just one slot.To find the relation between $p^+_{n,r}$ and $p_{n,r}(1)$ we apply Bayes' theorem, which gives
\begin{eqnarray}
p^+_{n,r} &=& Pr \{\; n \;\text{customers in the queue and}\; r \; \text{with the server just prior to the departure epoch}\nn\\
&&\text{and a batch arrives $\mid$ total number of customers in the queue just prior to the service}\nn\\
&&\text{completion epoch}\}, \quad n \geq 0, \quad a \leq r \leq b,\nn\\
&=&
\left\{\begin{array}{r@{\mskip\thickmuskip}l}
& \frac{1}{\tau}\left[\bar \la p_{0,r}(1)\right],\\
& \frac{1}{\tau} \left[\bar \la p_{n,r}(1)+\la \sum_{i=1}^{n} g_{i}p_{n-i,r}(1) \right], ~~n\geq 1
\end{array}\right.\nn\eea
where
\begin{eqnarray}
\tau &=& Pr\{\text{total number of customers in the queue just prior to the service completion}\nn
\\&&\text{or vacation termination epoch}\}
= \sum_{m=0}^{\infty}\sum_{r=a}^{b}p_{m,r}(1)+\sum_{m=0}^{\infty}Q_{m}(1). \nn
\end{eqnarray}
Similarly the results for vacations can be proved.
\end{proof}
\begin{lem}\label{SElem2}
The value of $\tau$ is given by
\begin{eqnarray}
\tau &=& \sum_{m=0}^{\infty}\sum_{r=a}^{b}p_{m,r}(1) + \sum_{m=0}^{\infty} Q_{m}(1)
= \frac{1-\left(1-\delta_{p}\right)\sum_{n=0}^{a-1}p_{n,0}}{\omega}\label{eq19} \\
\text{where}\nn\\
\omega &=& \sum_{n=0}^{a-1}\left[p^+_{n}E(V)+\delta_{p} Q^+_{n}E(V)+\left(1-\delta_{p}\right)Q^+_{n}\left\{\sum_{j=n}^{a-1}e_{j,n}\left(\sum_{i=a}^{b}g_{i-j}s_{i}+\sum_{i=b+1-j}^{\infty}g_{i}s_{b}\right)\right\}\right]
\nn\\&&+\sum_{n=a}^{b}\left(p^+_{n}+Q^+_{n}\right)s_{n} + \sum_{n=b+1}^{\infty}\left(p^+_{n}+Q^+_{n}\right)s_{b}\label{eq20}\\
\text{and}\nn\\
e_{n,i}&=&\sum_{j=i+1}^{n-1}e_{n,j}g_{j-i}+g_{n-i}, \quad i=0,1,\ldots,n-2 \quad \text{with}\quad e_{n,n-1}=g_{1}\quad \text{and} \quad e_{n,n}=1 \label{eq21}
\end{eqnarray}
\end{lem}
\begin{proof}[\textbf{Proof.}]
For single vacation substituting $\delta_{p}=0$ in (\ref{eq1}) and (\ref{eq2}), we obtain
\begin{eqnarray}
\lambda p_{n,0} = \bar{\lambda}e_{n,0}Q_{0}(1) + \sum_{m=1}^{n}e_{n,m}\bigg\{\bar{\lambda}Q_{m}(1) + \lambda \sum_{i=1}^{m}g_{i}Q_{m-i}(1)\bigg\}, \quad 1\leq n \leq a-1\label{eq22}
\end{eqnarray}
Using (\ref{eq22}) in (\ref{eq9}) and then summing  $r$ from $a$ to $b$, and $n$ from $0$ to $\infty$ and in the equations (\ref{eq9}) - (\ref{eq14}), we get
\begin{small}
\begin{eqnarray}
&&\left(\dfrac{z-1}{z}\right)\sum_{n=0}^{\infty}\left\{\sum_{r=a}^{b}p^*_{n,r}(z) + Q^*_{n}(z)\right\}\nn\\
&=& \left[\bar{\lambda}Q_{0}(1) + \sum_{n=1}^{a-1}\left\{\bar{\lambda}Q_{n}(1) + \lambda\sum_{i=1}^{n}g_{i} Q_{n-i}(1)\right\}\right]\delta_{p}V^*(z) \nonumber\\ &&+ \sum_{r=a}^{b}\left[\bar{\lambda}p_{0,r} + \sum_{n=1}^{a-1}\left\{\bar{\lambda}p_{n,r}(1) + \lambda\sum_{i=1}^{n}g_{i} p_{n-i,r}(1)\right\}V^*(z) +\sum_{n=a}^{b}\biggl\{\bar{\lambda}p_{n,r}(1) + \lambda\sum_{i=1}^{n}g_{i} p_{n-i,r}(1)\biggr\}S^*_{n}(z)\right] \nonumber\\ &&+\sum_{n=b+1}^{\infty}\left[\sum_{r=a}^{b}\left\{\bar{\lambda}p_{n,r}(1) + \lambda\sum_{i=1}^{n} g_{i} p_{n-i}(1)\right\}+\bar{\lambda}Q_{n}(1) + \lambda\sum_{i=1}^{n}g_{i} Q_{n-i}(1)\right]S^*_{b}(z)\nn\\
&&+\left(1-\delta_{p}\right)\sum_{i=0}^{a-1}\left(\sum_{r=a}^{b}g_{r-i}S^*_{r}(z)y^{r} +\sum_{n=1}^{\infty}g_{n+b-i}S^*_{b}(z)x^{n}y^{b}\right)\left(\bar{\lambda}e_{i,0}Q_{0}(1)+ \sum_{m=1}^{i}e_{i,m}\biggl\{\bar{\lambda}Q_{m}(1)+\lambda \sum_{j=1}^{m}g_{j}Q_{m-j}(1)\biggr\}\right)\nn\\
&& -\sum_{n=0}^{\infty}\sum_{r=a}^{b}p_{n,r}(1) - \sum_{n=0}^{\infty}Q_{n}(1) \nn	
\end{eqnarray}
\end{small}
Now by taking limit $z \rightarrow 1$ in both sides of above equation, by using L'Hospital's rule, and the normalizing condition $\left(1-\delta_{p}\right)\sum_{n=0}^{a-1}p_{n,0} + \sum_{n=0}^{\infty}\sum_{r=a}^{b}p_{n,r} + \sum_{n=0}^{\infty}Q_{n} = 1$ we obtain the desired result.
\end{proof}

\section{Joint distribution of queue length and number in a served batch at departure epoch}\label{JDD}
The primary intention of this section is to provide complete queue and server content  distribution together. Our foremost focal is to derive the bivariate pgf of the joint distribution at service completion epoch of a batch. Our aim is also to extract those joint distributions in a simple and elegant way. \\
We first define the following pgfs as:
\begin{eqnarray}
P(z, x, y) &=& \sum_{n=0}^{\infty}\sum_{r=a}^{b}p^*_{n,r}(z)x^{n}y^{r}, \quad  |x| \leq 1, \quad  |y| \leq 1\label{eq23}\\
P^+(x, y) &=& \sum_{n=0}^{\infty}\sum_{r=a}^{b}p^+_{n,r}x^{n}y^{r}, \quad |x| \leq 1, \quad |y| \leq 1\label{eq24}\\
P^+(x, 1) &=& \sum_{n=0}^{\infty}\sum_{r=a}^{b}p^+_{n,r}x^{n}=\sum_{n=0}^{\infty}p^+_{n}x^{n} = P^+(x), \quad |x| \leq 1\label{eq25}\\
Q^+(x) &=& \sum_{n=0}^{\infty}Q^+_{n}x^{n}, \quad |x| \leq 1\label{eq26}
\end{eqnarray}
Since our major concern is to find the bivariate pgf, multiplying (\ref{eq9}) - (\ref{eq11}) by appropriate powers of $x$ and $y$, summing over $n$ from $0$ to $\infty$ and $r$ from $a$ to $b$, and using (\ref{eq23}), we get
\begin{small}
\begin{eqnarray}
&&\left\{\dfrac{z-\left(\bar{\lambda}+\lambda G(x)\right)}{z}\right\} P(z, x, y)\nn\\
 &=& \sum_{n=a}^{b}\left[\sum_{r=a}^{b}\biggl\{\bar{\lambda}p_{n,r}(1) + \lambda\sum_{i=1}^{n}g_{i} p_{n-i,r}(1)\biggr\}+\left\{\bar{\lambda}Q_{n}(1)+\lambda\sum_{i=1}^{n} g_{i} Q_{n-i}(1)\right\}\right]S^*_{n}(z) y^{n} \nonumber\\
 &&+ \sum_{n=b+1}^{\infty}\left[\sum_{r=a}^{b}\biggl\{\bar{\lambda}p_{n,r}(1) + \lambda \sum_{i=1}^{n} g_{i}p_{n-i,r}(1)\biggr\}+ \biggl\{\bar{\lambda}Q_{n}(1) + \lambda \sum_{i=1}^{n}g_{i}Q_{n-i}(1)\biggr\}\right]S^*_{b}(z)x^{n-b}y^{b}\nn\\
 &&+\left(1-\delta_{p}\right)\sum_{i=0}^{a-1}\bigg[\bar{\lambda}e_{i,0}Q_{0}(1)+ \sum_{m=1}^{i}e_{i,m}\biggl\{\bar{\lambda}Q_{m}(1)+\lambda \sum_{j=1}^{m}g_{j}Q_{m-j}(1)\biggr\}\bigg]\left(\sum_{r=a}^{b}g_{r-i}S^*_{r}(z)y^{r} +\sum_{n=1}^{\infty}g_{n+b-i}S^*_{b}(z)x^{n}y^{b}\right)\nn\\
 &&-\bar{\lambda}\sum_{r=a}^{b}p_{0,r}(1)y^{r} -\sum_{n=1}^{\infty}\sum_{r=a}^{b}\biggl\{\bar{\lambda}p_{n,r}(1) + \lambda \sum_{i=1}^{n}g_{i} p_{n-i,r}(1)\biggr\}x^{n}y^{r} \label{eq27}
\end{eqnarray}
\end{small}
Now substituting $z= \bar{\lambda}+\lambda G(x)$ in (\ref{eq27}) and using (\ref{eq15}) - (\ref{eq18}), and (\ref{eq23}) we obtain
\begin{eqnarray}
P^+(x, y) &=& \left(1-\delta_{p}\right)\sum_{n=0}^{a-1}Q^+_{n}\left[\sum_{j=n}^{a-1}e_{j,n}\left(
\sum_{i=a}^{b}g_{i-j}K^{(i)}(x)y^{i}+\sum_{i=b+1-j}^{\infty}g_{i}x^{i+j-b}K^{(b)}(x)y^{b}\right)\right]\nn\\
&&+\sum_{n=a}^{b}\left(p^+_{n}+Q^+_{n}\right)K^{(n)}(x)y^{n}+
\sum_{n=b+1}^{\infty}\left(p^+_{n}+Q^+_{n}\right)K^{(b)}(x)x^{n-b}y^{b}\label{eq28}
\end{eqnarray}
where $K^{(r)}(x)=S^*_{r}(\bar{\lambda}+\lambda G(x))$ is the pgf of $k^{(r)}_{j}$,
$\left(a \leq r \leq b, ~ j \geq 0 \right)$ and\\
$k^{(r)}_{j} =  Pr\{j\; \text{arrivals during the service time of a batch of size} \; r\}$.\\
Now multiplying (\ref{eq12}) to (\ref{eq14}) with appropriate powers of $x$ and summing over $n$ from $0$ to $\infty$ we get
\begin{eqnarray}
\biggl\{\dfrac{z-\left(\bar{\lambda}+\lambda G(x)\right)}{z}\biggr\}
\sum_{n=0}^{\infty}Q^*_{n}(z)x^{n} &=& \sum_{m=a}^{b}\biggl\{\bar \la p_{0,m}(1)+\sum_{n=1}^{a-1}\biggl(\bar{\lambda}p_{n,m}(1)+\lambda \sum_{i=1}^{n}g_{i}p_{n-i,m}(1)\biggr)\biggr\}x^{n}V^{*}(z)\nonumber
\\&&+\delta_{p}\biggl\{\bar{\lambda}Q_{0}(1) + \sum_{n=1}^{a-1}\bigg(\bar{\lambda}Q_{n}(1) + \lambda \sum_{i=1}^{n}g_{i}Q_{n-i}(1)\bigg)\biggr\}x^{n}V^*(z)\nn\\
&&-\bar{\lambda}Q_{0}(1)-\sum_{n=1}^{\infty}\biggl\{\bar{\lambda}Q_{n}(1)+\lambda \sum_{i=1}^{n}g_{i}Q_{n-i}(1)\biggr\}x^{n}\label{eq29}
\end{eqnarray}
Putting $z=\bar{\lambda}+\lambda G(x)$ in (\ref{eq29}) and using (\ref{eq15}) - (\ref{eq18}), (\ref{eq26}), we get
\begin{eqnarray}
Q^{+}(x) = \sum_{n=0}^{a-1}p^+_{n}x^{n} H(x) + \delta_{p} \sum_{n=0}^{a-1}Q^+_{n}x^{n}H(x)\label{eq30}
\end{eqnarray}
where $H(x)= \sum_{j=0}^{\infty}h_{j}x^{j}= V^{*}(\bar{\lambda}+\lambda G(x))$ is the pgf of $h_{j}$ and\\
$h_{j} =  Pr\{j\; \text{arrivals during the vacation time of a server}\}$\\
Now substituting $y=1$ in (\ref{eq28}) and using (\ref{eq25}) and (\ref{eq30}) and after some algebraic manipulation we obtain\\
\begin{eqnarray}
P^+(x)= \frac{\begin{aligned}
\sum_{n=0}^{a-1}p^+_{n}x^{n}\biggl\{H(x)-1\biggr\}K^{(b)}(x) + \sum_{n=0}^{a-1}
Q^+_{n}\left[x^{n}\biggl\{\delta_{p}H(x)-1\biggr\}K^{(b)}(x)\right.\\
\left.+x^{b}\left(1-\delta_{p}\right)\left\{\sum_{j=n}^{a-1}e_{j,n}\left(
\sum_{i=a}^{b}g_{i-j}K^{(i)}(x)+\sum_{i=b+1-j}^{\infty}g_{i}x^{i+j-b}K^{(b)}(x)
\right)\right\}\right]\\
+\sum_{n=a}^{b-1}\bigg(p^+_{n}+Q^+_{n}\bigg)\biggl\{x^{b}K^{(n)}(x)-x^{n}K^{(b)}(x)
\biggr\}\end{aligned}}{x^b-K^{(b)}(x)} \label{eq31}
\end{eqnarray}
where $e_{j,n}$ is given in (\ref{eq21})\\
The expression represented by (\ref{eq31}) is the pgf of only queue length distribution which is essential to calculate the tail distribution as it is a powerful tool for computing cell loss ratio in ATM switching network. The result is not available in literature to the best of authors' knowledge. Thus we hope that this leads to a new contribution in the field of queueing literature.\\
It should be noted here that one can not find the server content from (\ref{eq31}). Now looking back into (\ref{eq27}), using (\ref{eq29}), (\ref{eq31}) and
 after some algebraic simplification, we get
\begin{eqnarray}
P^+(x,y)= \frac{\begin{aligned}
\sum_{n=0}^{a-1}p^+_{n}x^{n}\biggl\{H(x)-1\biggr\}y^{b}K^{(b)}(x) +
\sum_{n=0}^{a-1}
Q^+_{n}\left[x^{n}\biggl\{\delta_{p}H(x)-1\biggr\}y^{b}K^{(b)}(x)\right.\\
\left.+\left(1-\delta_{p}\right)\left\{\sum_{j=n}^{a-1}e_{j,n}\bigg(
\sum_{i=a}^{b}g_{i-j}K^{(i)}(x)\biggr\{\left(x^{b}-K^{(b)}(x)\right)y^{i}+y^{b}K^{(b)}(x)\biggr\}\right.\right.\\\left.\left.+\sum_{i=b+1-j}^{\infty}g_{i}x^{i+j}y^{b}K^{(b)}(x)
\bigg)\right\}\right]\\
+\sum_{n=a}^{b-1}\left(p^+_{n}+Q^+_{n}\right)\left[ K^{(b)}(x)K^{(n)}(x) \left(y^b-y^n\right)+\left(K^{(n)}(x)x^b y^n-K^{(b)}(x)x^n y^b\right)\right]
\end{aligned}}{x^b-K^{(b)}(x)} \label{eq32}
\end{eqnarray}
The above expression represents the bivariate pgf of queue and server content distribution at service completion epoch, which is the key ingredient of our analysis. No such result is available in literature so far to the best of authors' knowledge. The joint distribution can be found by inverting the bivariate pgf which discussed in later analysis.
\subsection{Steady-state conditions and determination of unknowns in the numerator of (\ref{eq32})}
\hspace*{0.3cm} Although the stability of the concerned queue is assumed in model description,it can also be proved using the stability condition given by Abonikov and Dukhovny. We conclude that the corresponding Markov chain is ergodic if and only if $\frac{d}{dx}K^{(b)}(x)\mid_{x=1} \textless \,b$. Consequently, the steady state distribution exists if $\frac{\lambda \bar{g}}{\mu_{b}}\, \textless \, b$, i.e $\frac{\lambda \bar{g}}{b \mu_{b}} = \rho \, \textless \, 1$.\\
\hspace*{0.3cm}Before extracting the joint probabilities from bi-variate pgf represented by (\ref{eq32}), we need to find the unknown quantities $p^+_{n}$, $\left(0 \leq n \leq b-1\right)$ and $Q^+_{n}$, $\left(0 \leq n \leq a-1 \right)$, appearing in the numerator. Since both (\ref{eq31}) and (\ref{eq32}) contains same unknown quantities, we use (\ref{eq31}) to find the unknowns, without any loss of generality. The procedure of finding these unknowns although, available in literature, we describe it briefly for sake of completeness. We have to determine $a+b$ unknowns here. We use the result below given in the form of lemma,which represents between $Q^+_{n}$ in terms of $p^+_{n}$, that leads us to $b$ unknowns. From Rouch$\acute{\mbox{e}}$'s theorem it is known that $D(x)= x^{b} - K^{(b)}(x)$, (say) has exactly $b$ zeroes in the unit circle $|x| \leq 1$. Using these zeroes and the normalizing condition, we obtain $b$ linear simultaneous equations and hence solving them, we get the unknowns.
\begin{lem}
The relation between $p^+_{n}$ and $Q^+_{n}$ are given by:
\begin{eqnarray}
Q^+_{n} &=& \sum_{i=0}^{n}\xi_{i}p^+_{n-i}, \quad n=0,1,2,\ldots,a-1 \nn\\
\xi_{0} &=& \frac{h_{0}}{1-h_{0}}\nn\\
\xi_{n} &=& \frac{1}{1-h_{0}}\left[h_{n} + \sum_{i=1}^{n}h_{i}\xi_{n-i}\right] \quad n=1,2,\ldots,a-1\nn
\end{eqnarray}
\end{lem}
\subsection{Extracting probabilities from the known bivariate pgf}
\hspace*{0.3cm} Now we have the completely known bivariate pgf, the key ingredient of this analysis. In this section we now proceed to extract the joint probabilities using partial fraction technique. First we accumulate the the coefficients of $y^i$, $\left(a \leq r \leq b\right)$ from the bi-variate pgf (\ref{eq32}) and are given below.
\\
Coefficient of $y^i$, $a \leq i \leq b-1$:
\begin{eqnarray}
\sum_{n=0}^{\infty}p^+_{n,i}x^{n} = \left(1-\delta_{p}\right)\sum_{n=0}^{a-1}Q^+_{n} \sum_{j=n}^{a-1}e_{j,n}g_{i-j}K^{(i)}(x) + \left(p^+_{i}+Q^+_{i}\right)K^{(i)}(x) \label{eq33}
\end{eqnarray}
Coefficient of $y^b$:
\begin{eqnarray}
\sum_{n=0}^{\infty}p^+_{n,b}x^{n} &=& \frac{1}{x^{b}-K^{(b)}(x)}\left[\sum_{n=0}^{a-1} p^+_{n}x^{n}\bigg(H(x)-1\bigg) K^{(b)}(x)+\left(1-\delta_{p}\right)K^{(b)}(x)\right.\nn\\
&&\left.\sum_{n=0}^{a-1}Q^+_{n}\sum_{j=n}^{a-1}e_{j,n}\left(g_{b-j}K^{(i)}(x) +\sum_{i=b+1-j}^{\infty}g_{i} x^{i+j}K^{(b)}(x)\right)\right]\nn\\
&&+\left(1-\delta_{p}\right)\sum_{n=0}^{a-1}Q^+_{n}\sum_{j=n}^{a-1}e_{j,n}g_{b-j} K^{(b)}(x) + \sum_{n=a}^{b-1}\left(p^+_{n}+Q^+_{n}\right)x^{n-b}K^{(b)}(x) \label{eq34}
\end{eqnarray}
Collecting the coefficients of $x^n$ from both sides of (\ref{eq33}) and we obtain the joint distribution as:
\begin{eqnarray}
p^+_{n,i} = \bigg[\left(1-\delta_{p}\right)\sum_{n=0}^{a-1}Q^+_{n}\sum_{j=n}^{a-1}e_{j,n}g_{i-j}  + \left(p^+_{i}+Q^+_{i}\right)\bigg]k_n^{(i)}, \quad n \geq 0, \, a \leq i \leq b-1
\label{eq35}
\end{eqnarray}
Now we are left only with $p^+_{n,b}$, for which we need to invert the right hand side of (\ref{eq34}). If service and vacation time distribution are known then right hand side of  (\ref{eq35}) is completely known function of $x$. The inversion process is described in subsequent analysis.\\
\hspace*{0.3cm} After substituting $K^{(r)}(x)=S^*_r(\bar \la + \la G(x)),~a\leq r \leq b,$ and $H(x)=V^*(\bar \la + \la G(x))$ in the expression (\ref{eq34}),
let us denote numerator of (\ref{eq34}) as $\Lambda(x)$ and denominator as $D(x)$. The degrees of $\Lambda(x)$ and $D(x)$ depend on the distributions of service/vacation time and service threshold value `$a$' and maximum capacity `$b$'. Let the degree of numerator be $L_1$ and that of denominator be $M_1$, respectively. Now, we obtain $p_{n,b}^+$ in terms of roots of $D(x)=0$ by discussing the following cases:

\begin{itemize}
	\item \textbf{When all the zeroes of $D(x)$ in $|x|>1$ are distinct}\\
	Let us assume $\alpha_{1},\alpha_{2},\ldots,\alpha_{M_1}$ to be the roots of $D(x)=0$, out of which $M_1-b$ are the distinct roots in $|x| > 1$, as the degree of $D(x)$ is $M_1$ and $b$ insides roots are cancelled out with the roots of the numerator. It is also assumed that $D(x)=0$ has $b$ simple roots inside the unit circle. \\
	~~\\
	\emph{Case 1}: $L_1\geq M_1$\\
	\hspace*{0.5cm}Applying the partial-fraction expansion, the rational function $\sum_{n=0}^{\infty}p_{n,b}^+x^n$ can be uniquely written as
	\bea \sum_{n=0}^{\infty}p_{n,b}^+x^n=\sum_{i=0}^{L_1-M_1}\tau_{i} x^i +\sum_{k=1}^{M_1-b}\frac{c_{k}}{\alpha_{k}-x},\label{eq36}\eea
	for some constants $\tau_{i}$ and $c_{k}$'s. The first sum is the result of the division of the polynomial $\Lambda(x)$ by $D(x)$ and the constants $\tau_{i}$ are the coefficients of the resulting quotient. Using the residue theorem, we have
	$ c_{k}=-\frac{\Lambda(\alpha_{k})}{D^{\prime}(\alpha_{k})},\quad k=1,2,\ldots,M_1-b. $\\
	Now, collecting the coefficient of $x^n$ from both the sides of (\ref{eq36}), we have
	\bea p_{n,b}^+= \left\{\begin{array}{r@{\mskip\thickmuskip}l}
		&\tau_{n}+\sum\limits_{k=1}^{M_1-b}\frac{c_{k}}{\alpha^{n+1}_{k}},\quad 0\leq n\leq L_1-M_1,\\
		&\\
		& \sum\limits_{k=1}^{M_1-b}\frac{c_{k}}{\alpha^{n+1}_{k}},\quad n> L_1-M_1~.\end{array}\right.\label{eq37}\eea
	\emph{Case 2}: $L_1< M_1$\\
	Using partial-fraction technique on $\sum_{n=0}^{\infty}p_{n,b}^+ x^n$, we have
	\bea \sum_{n=0}^{\infty}p_{n,b}^+x^n=\sum_{k=1}^{M_1-b}\frac{c_{k}}{\alpha_{k}-x},\label{eq38}\eea
	where $c_{k}=-\frac{\Lambda(\alpha_{k})}{D^{\prime}(\alpha_{k})},\quad k=1,2,\ldots,M_1-b. $\\
	Now, collecting the coefficient of $x^n$ from both the sides of (\ref{eq38}), we obtain
	\bea p_{n,b}^+ &=& \sum\limits_{k=1}^{M_1-b}\frac{c_{k}}{\alpha^{n+1}_{k}},\quad n\geq 0. \label{eq39}\eea
	
	\item \textbf{When some zeroes of $D(x)$ in $|x|>1$ are repeated}\\
	The denominator $D(x)$ may have some multiple zeroes with absolute value greater than one. Let $D(x)$ has total $m$ multiple zeroes, say $\beta_{1},\beta_{2},\ldots,\beta_{m}$ with multiplicity $\pi_{1},\pi_{2},\ldots,\pi_{m}$, respectively. Further it is clear that $D(x)$ has total $(M_1-b-\chi)$ distinct zeroes, where $ \chi=\sum_{i=1}^{m}\pi_{i}$, say, $\alpha_{1},\alpha_{2},\ldots,\alpha_{M_1-b-\chi}$. It is also assumed that $D(x)$ has $b$ simple zeroes inside the unit circle.\\
	~~\\
	\emph{Case 1}: $L_1\geq M_1$\\
	Applying the partial-fraction method, we can uniquely write $\sum_{n=0}^{\infty}p_{n,b}^+x^n,$ as
	\bea \hspace*{-1.0cm}\sum_{n=0}^{\infty}p_{n,b}^+x^n=\sum_{i=0}^{L_1-M_1}\tau_{i} x^i+\sum_{k=1}^{M_1-b-\chi}\frac{c_{k} }{\alpha_{k}-x}+\sum_{\nu=1}^{m}\sum_{i=1}^{\pi_{\nu}}\frac{\zeta_{\nu,i}}{(\beta_{\nu}-x)^{\pi_{\nu}-i+1}} \label{eq40}. \eea
	Using residue theorem, we obtain
	\begin{small}
		\bea \hspace*{-1.2cm}c_{k}&=&-\frac{\Lambda(\alpha_{k})}{D^{\prime}(\alpha_{k})},\quad k=1,2,\ldots,M_1-b-m, \nn \\
		\hspace*{-1.2cm}\zeta_{\nu,i}&=&
		\frac{1}{(\pi_{\nu}-i)!}~\lim_{x \rightarrow \beta_{\nu}}~\frac{d^{(\pi_{\nu}-i)}}{dx^{(\pi_{\nu}-i)}}\left[ \frac{(\beta_{\nu}-x)^{\pi_{\nu}}\Lambda(x)}{D(x)}\right],~ \nu=1,2,\ldots,m, ~ i=1,2,\ldots,\pi_{\nu}.\nn \eea
	\end{small}
	Now collecting the coefficient of $x^n$ from both the sides of (\ref{eq40}), we have the distribution $p_{n,b}^+~,~n\geq 0$ as:
	\begin{small}
		\bea \hspace*{-1cm}p_{n,b}^+=\displaystyle \left\{\begin{array}{r@{\mskip\thickmuskip}l}
			&\tau_{n}+\sum\limits_{k=1}^{M_1-b-\chi}\frac{c_{k}}{\alpha^{n+1}_{k}}+\displaystyle
			\sum_{\nu=1}^{m}\sum_{i=1}^{\pi_{\nu}}
			\binom{\pi_{\nu}+n-i}{\pi_{\nu}-i}\frac{\zeta_{\nu, i}}{\beta_{\nu}^{\pi_{\nu}+n+1-i}},~~ 0\leq n\leq L_1-M_1,\\
			& \\
			& \sum\limits_{k=1}^{M_1-b-\chi}\frac{c_{k}}{\alpha^{n+1}_{k}}+\displaystyle
			\sum_{\nu=1}^{m}\sum_{i=1}^{\pi_{\nu}}
			\binom{\pi_{\nu}+n-i}{\pi_{\nu}-i}\frac{\zeta_{\nu, i}}{\beta_{\nu}^{\pi_{\nu}+n+1-i}},~~ n> L_1-M_1.\end{array}\right. \label{eq41}\eea
	\end{small}
	~~\\
	\emph{Case 2}: $L_1< M_1$\\
	Here, in partial-fraction, we omit only  the first summation term of the right hand side of (\ref{eq41}). Then collecting the coefficients of $x^n$, we obtain $p_{n,b}^+$, which are given by
	\begin{small}
		\bea \hspace*{-1cm}p_{n,b}^+=\sum\limits_{k=1}^{M_1-b-\chi}\frac{c_{k}}{\alpha^{n+1}_{k}}+\sum_{\nu=1}^{m}\sum_{i=1}^{\pi_{\nu}}
		\binom{\pi_{\nu}+n-i}{\pi_{\nu}-i}\frac{\zeta_{\nu, i}}{\beta_{\nu}^{\pi_{\nu}+n+1-i}},~ n\geq 0.\label{eq42}\eea
	\end{small}
\end{itemize}
This completes the extraction of the joint probabilities $p_{n,b}^+$.
\begin{remark}
Sometimes it may be required to know the probabilities at vacation termination epoch i.e., $Q^+_{n}$, $n \geq 0$, which can be easily extracted from $Q^+(x)$ by comparing the coefficients of $x^n$ and are given as:

 \begin{align} Q^+_{n} = \left\{\begin{array}{r@{\mskip\thickmuskip}l}
\sum_{i=0}^{n} \left(p_{i}^+ + \delta_{p} Q_{i}^+\right),~~& 1 \leq n \leq a-1,\\
\sum_{i=0}^{a-1} \left(p_{i}^+ + \delta_{p} Q_{i}^+\right), ~~& n \geq a.
\end{array}\right.
\end{align}\label{eq43}

\end{remark}
\begin{remark}
One should note that the queue length distribution at service completion epoch i.e, $p^+_{n}$ can be found by summing the joint probabilities i.e, $p^+_{n,r}$ over $r$ from $a$ to $b$. One does not need to extract the the probabilities further from $P^+(x)$.
\end{remark}
\hspace*{0.3cm} Thus we now have joint distribution of queue length and server content at service termination epoch as well as only queue length distribution at vacation termination epoch. in next section, the probability distribution at arbitrary slot is obtained.
\section{Joint distribution of queue and server content at arbitrary epoch}\label{JDA}
From the joint distribution at arbitrary epoch, one can compute several marginal distribution and performance measures. We now generate a relationship between the joint probabilities at service completion and arbitrary epochs.
\begin{thm}
The probabilities at service completion and arbitrary epoch are related by
\begin{eqnarray}
p_{0,0} &=& \dfrac{Q^+_{0}}{E^*}\label{eq44}\\
p_{n,0} &=& \dfrac{1}{E^*}\sum_{m=0}^{n}e_{n,m}Q^+_{m},\quad 1 \leq n \leq a-1\label{eq45}\\
p_{0,r} &=& \dfrac{1}{E^{*}}\biggl(\left(1-\delta_{p}\right)\sum_{i=0}^{a-1}g_{r-i}p_{i,0}+p^+_{r}+ Q^+_{r} - p^+_{0,r}\biggr), \quad a \leq r \leq b\label{eq46}\\
p_{n,r} &=& \sum_{i=1}^{n}g_{i}p_{n-i,r} - \dfrac{p^+_{n,r}}{E^*},\quad n \geq 1, \quad a
\leq r \leq b-1 \label{eq47}\\
p_{n,b} &=& \dfrac{1}{E^{*}}\biggl(\left(1-\delta_{p}\right)\sum_{i=0}^{a-1}g_{n+b-i}p_{i,0} +p^+_{n+b}+ Q^+_{n+b} - p^+_{n,b}\biggr), \quad n \geq 1 \label{eq48}\\
Q_{0} &=& \dfrac{1}{E^*}\biggl(p^+_{0}-\left(1-\delta_{p}\right)Q^+_{0}\biggr)\label{eq49}\\
Q_{n} &=& \sum_{i=1}^{n}g_{i}Q_{n-i}+\dfrac{1}{E^*}\biggl(p^+_{n}-\left(1-\delta_{p}\right)Q^+_{n}\biggr), \quad 1 \leq n \leq a-1\label{eq50}\\
Q_{n} &=& \sum_{i=1}^{n}g_{i}Q_{n-i} - \dfrac{Q^+_{n}}{E^*}, \quad n \geq a\label{eq51}
\end{eqnarray}
where $E^* = \lambda \omega + \sum_{n=0}^{a-1}\sum_{j=n}^{a-1}e_{j,n} Q^+_{n}$ and $e_{j,n}$ and $\omega$ are given in lemma \ref{SElem2}
\end{thm}
\begin{proof}[\textbf{Proof}]
For single vacation putting $\delta_{p}=0$ in equation (\ref{eq1}),(\ref{eq2}) and
then dividing by $\tau$, we obtain using (\ref{eq19})
\begin{eqnarray}
p_{0,0} &=& \dfrac{\biggl(1-\left(1-\delta_{p}\right)\sum_{i=0}^{a-1}p_{i,0}\biggr)}{\lambda \omega} Q^+_{0}\label{eq52}\\
p_{n,0} &=& \dfrac{\biggl(1-\left(1-\delta_{p}\right)\sum_{i=0}^{a-1}p_{i,0}\biggr)}{\lambda \omega} \sum_{m=0}^{n} e_{n,m}Q^+_{m}\label{eq53}
\end{eqnarray}
Now dividing (\ref{eq52}) and (\ref{eq53})
\begin{eqnarray}
p_{n,0} &=& \dfrac{p_{0,0}}{Q^+_{0}} \sum_{m=0}^{n}e_{n,m} Q^+_{m}, \quad 1 \leq n \leq a-1
\label{eq54}
\end{eqnarray}
Using (\ref{eq54}) in (\ref{eq52}),(\ref{eq53})
\begin{eqnarray}
p_{0,0} &=&\dfrac{Q^+_{0}}{\lambda \omega +\left(1-\delta_{p}\right) \sum_{n=0}^{a-1}\sum_{j=n}^{a-1}e_{j,n} Q^+_{n}}\label{eq55}\\
p_{n,0} &=& \dfrac{\sum_{m=0}^{n}e_{n,m} Q^+_{m}}{\lambda \omega + \left(1-\delta_{p}\right)\sum_{n=0}^{a-1}\sum_{j=n}^{a-1}e_{j,n} Q^+_{n}}, \quad 1 \leq n \leq a-1\label{eq56}
\end{eqnarray}
Putting $z=1$ in (\ref{eq9}) and then dividing by $\tau$ we obtain
\begin{eqnarray}
p_{0,r} = \dfrac{\left(1- \left(1-\delta_{p}\right)\sum_{i=0}^{a-1}p_{i,0}\right)}{\lambda \omega}
\bigg(\left(1-\delta_{p}\right)\sum_{i=0}^{a-1}g_{r-i}p_{i,0}+p^+_{r}+Q^+_{r}-p^+_{0,r} \bigg)\label{eq57}
\end{eqnarray}
Now from (\ref{eq52}) and (\ref{eq55}) gives
\begin{eqnarray}
1-\left(1-\delta_{p}\right)\sum_{i=0}^{a-1}p_{i,0}=\dfrac{\lambda \omega}
{\lambda \omega + \left(1-\delta_{p}\right)\sum_{n=0}^{a-1}\sum_{j=n}^{a-1}e_{j,n}Q^+_{n}}\label{eq58}
\end{eqnarray}
Thus using (\ref{eq58}) in (\ref{eq57}) we get (\ref{eq46})\\
Again putting $z=1$ in (\ref{eq10}) to (\ref{eq14}) and by similar approach, we obtain
(\ref{eq47}) to (\ref{eq48})
\end{proof}

\section{Marginal distributions and performance measures}\label{MaD}
Due to applicability of the  concerned model, the marginal distributions and performance measures are the most important ingredients for system designer/vendors. Here the marginal distributions in terms of known distribution are described as follows:
\begin{itemize}
	\item  the distribution of the number of customers in the system (which includes total the number of customers in the queue and with the server), $p_n^{sys},~~n\geq 0$,
	\begin{align*}
	p_n^{sys} = \left\{\begin{array}{r@{\mskip\thickmuskip}l}
	&\left(1-\delta_{p}\right)p_{n,0} +Q_n ~~,\hspace{2.0cm}0\leq n \leq a-1\\
	& \sum_{r=a}^{min(b,n)}p_{n-r,r}+Q_n~~,\hspace{1.3cm} a\leq n \leq b,\vspace{0.1cm}\\
	&\sum_{r=a}^{b}p_{n-r,r}\hspace{3.3cm} n\geq b+1.
	\end{array}\right.
	\end{align*}
	
	\item the distribution of the number of customers in the queue $p_n^{queue},~~n\geq 0$,
	\begin{align*}
	p_n^{queue} =\left\{\begin{array}{r@{\mskip\thickmuskip}l}
	& \left(1-\delta_{p}\right)p_{n,0}+\sum_{r=a}^{b}p_{n,r}+Q_n~,\hspace{1.3cm}0\leq n \leq a-1\\
	&\sum_{r=a}^{b}p_{n,r}+Q_n\hspace{4.0cm} n\geq a.
	\end{array}\right.
	\end{align*}
	
	\item the probability that the server is in busy state ($P_{busy}$), in the vacation state ($Q_{vac}$), and in the dormant state ($P_{dor}$) are given by $P_{busy}=\sum_{n=0}^{\infty}\sum_{r=a}^{b}p_{n,r}$,~ $Q_{vac}=\sum_{n=0}^{\infty}Q_n$,~ $P_{dor}=\sum_{n=0}^{a-1}p_{n,0}$.
	
	\item the conditional distribution of the number of customers undergoing service with the server given that the server is busy,
	$p_r^{ser}$, ($a \leq r \leq b $)
	\begin{align*}
	p_r^{ser} = \sum_{n=0}^{\infty}p_{n,r}/P_{busy} ~~(a\leq r \leq b).
	\end{align*}
	
\end{itemize}
\hspace*{0.3cm}Performance measures are very crucial from the application point of view as it improves the efficiency of the system. Some performance measures
can be derived from the pgf, and some can be obtained using completely known state probabilities and marginal distributions,
in a very simplified manner are presented below:
\begin{itemize}
	\item average number of customers waiting in the queue at any arbitrary time ($L_q)=\sum_{n=0}^{\infty}n p_n^{queue}$,
    \item average number of customers waiting in the queue service completion epoch ($L_q^+$)
    \begin{small}
\begin{eqnarray}
 &&2 \left(b - K^{'(b)}(1)\right)L^+_{q}\nn\\
 &&= \sum_{n=0}^{a-1}\left(p^+_{n} + \delta_{p} Q^+_{n}\right)\left[H^{''}(1) + 2n H^{'}(1) + 2  H^{'}(1)K^{'(b)}(1) + n K^{'(b)}(1)\right]\nn\\
	&&- \sum_{n=0}^{a-1} Q^+_{n}\left(1 - \delta_{p}\right)\left(n(n-1) + nK^{'(b)}(x) + K^{''(b)}(1)\right)\nn\\
	&&+\left(1-\delta_{p}\right)\sum_{n=0}^{a-1}\sum_{j=n}^{a-1}e_{j,n} \left[\sum_{i=a}^{b}\left\{K^{''(i)}(1) + 2bK^{'(i)}(1) + b(b-1)\right\}\right.\nn\\
&&\left.+\sum_{i=b+j-1}^{\infty}g_{i}\left\{K^{''(b)}(1) + 2(i+j)K^{'(b)}(1) + (i+j)(i+j-1)\right\}\right]\nn\\
	&&+ \sum_{n=a}^{b-1}\left(p^+_{n}+Q^+_{n}\right)\left[b(b-1) + 2bK^{'(n)}(1) + K^{''(n)}(1)-\left\{n(n-1) + 2nK^{'(b)}(1) + K^{''(b)}(1)\right\}\right] + K{''(b)}(1) -b(b-1)\nn
\end{eqnarray}
\end{small}
	\item average number in the system ($L)=\sum_{n=0}^{\infty}n p_n^{sys}$,
	\item average number of customers with the server ($L_s)=\sum_{r=a}^{b}r p_r^{ser}$,
	\item average waiting time of a customer in the queue $(W_q)=\frac{L_q}{\la \bar{g}}$ as well as in the system $(W)=\frac{L}{\la \bar{g}}$.
\end{itemize}

\section{Numerical illustration}\label{NR}
The main objective of this section is to bring out the qualitative aspects of the concerned queue from the implementation point of view through interesting numerical examples which can be fruitful to real-life scenarios. We carry out extensive numerical results by evoking the service or vacation time distributions as discrete phase type (PH$_D$), negative binomial (NB) and geometric (Geo) distributions. All the calculations were performed using Maple 15 on PC having
configuration Intel (R) Core (TM) i5-3470 CPU Processor @ 3.20 GHz with 4.00 GB of RAM, and the results are presented in the tabular form.\\
\hspace*{0.3cm}It is worthwhile to mention here that in most of the practical circumstances, the number of customers in an arriving group is finite. As a consequence, for the computational purpose we consider arriving group size to be of finite support.\\
\textbf{Example 1:} Due to the several inescapable reasons such as bad weather, electrical faults, fading channel etc. we have to usually deal with the transmission errors in a wireless communication channel. Discrete phase type distribution (PH$_D$) plays a noteworthy role to control this error. It also covers a wide range of almost all relevant discrete distributions. Keeping this in mind, in this example we consider service time to follow discrete phase type (PH$_D$) distribution with the representation ($\boldsymbol{\beta}_i,~\textbf{T}_i$) for $a\leq i \leq b$, where $\boldsymbol{\beta}_i$ is a row vector of order $\nu$ and $\textbf{T}_i$ is a square matrix of order $\nu$. The associated  pgf is $z \boldsymbol{\beta}_i\left(I-z\textbf{T}_i\right)^{-1}\eta_i$ with $\eta_i+\textbf{T}_ie=e$, where $I$ is the identity matrix of appropriate order and $e$ is the column vector. Consequently,
$K^{(i)}(x)=(\bar \la+\la G(x)) \boldsymbol{\beta}_i\left(I-(\bar \la+\la G(x))\textbf{T}_i\right)^{-1}\eta_i$.  Here we consider $a=8,~b=15,~\la=0.055, g_1=0.25, g_2=0.25, g_3=0.25, g_4=0.25$, and the matrices for PH$_D$ distribution for different batch-size is taken as $\textbf{T}_r=\left(
                                                                                                                       \begin{array}{ccc}
                                                                                                                         1-\theta_r & \theta_r & 0 \\
                                                                                                                         0 & 1-\theta_r & \theta_r \\
                                                                                                                          0& 0 & 1-\theta_r \\
                                                                                                                       \end{array}
                                                                                                                     \right)$
with $\theta_8=0.2, \theta_9=0.1, \theta_{10}=0.066667, \theta_{11}=0.05, \theta_{12}=0.04, \theta_{13}=0.033333, \theta_{14}=0.028571, \theta_{15}=0.025$ so that $s_8=8.75, s_9=17.5, s_{10}=26.25, s_{11}=35.0, s_{12}=43.75, s_{13}=52.5, s_{14}=61.25, s_{15}=70.0$ and $\rho=0.641666$.
The vacation time follows geometric distribution with pmf $(v_j)=\eta~(1-\eta)^{j-1},~~j\geq 1$ with $\eta=0.4$, pgf $=\dfrac{\eta z}{1-(1-\eta)z}$, mean ($\bar v$)=2.50. The joint distribution at service completion and arbitrary epochs along with some performance measures are displayed in Tables [\ref {tb1}-\ref {tb4}].\\
\newpage
\begin{table}[h!]
\begin{tiny}
\caption{Probability distribution at service completion epoch for single vacation}\lb{tb1}$\vspace{0.03cm}$
\begin{tabular}{|c|c|c|c|c|c|c|c|c|c|c|} \hline
$n$& $p_{n,8}^+$& $p_{n,9}^+$& $p_{n,10}^+$& $p_{n,11}^+$& $p_{n,12}^+$& $p_{n,13}^+$&$p_{n,14}^+$&$p_{n,15}^+$& $p_n^+$&$Q_n^+$\\\hline
0&0.130337&0.069044&0.038114&0.016857&0.016857&0.000870&0.000641&0.000476&0.257609&0.224887\\
1&0.012320&0.009892&0.006540&0.003195&0.003195&0.000183&0.000139&0.000501&0.033029&0.036390\\
2&0.013229&0.011105&0.007537&0.003744&0.003744&0.000218&0.000167&0.000549&0.036856&0.040845\\
3&0.014200&0.012461&0.008680&0.004385&0.000359&0.000359&0.000201&0.000620&0.041170&0.045857\\
4&0.015236&0.013974&0.013974&0.005133&0.000425&0.000312&0.000241&0.000713&0.046029&0.051495\\
5&0.004023&0.005769&0.004954&0.002810&0.000247&0.000247&0.000150&0.000725&0.018870&0.021795\\
10&0.000570&0.001740&0.002144&0.001506&0.000152&0.000128&0.000110&0.000746&0.007100&0.001226\\
25&0.000000&0.000013&0.000061&0.000095&0.000016&0.000020&0.000022&0.000413&0.000642&0.000000\\
50&0.000000&0.000000&0.000000&0.000000&0.000000&0.000000&0.000000&0.000090&0.000090&0.000000\\
75&0.000000&0.000000&0.000000&0.000000&0.000000&0.000000&0.000000&0.000017&0.000017&0.000000\\
100&0.000000&0.000000&0.000000&0.000000&0.000000&0.000000&0.000000&0.000003&0.000003&0.000000\\
$\geq$120&0.000000&0.000000&0.000000&0.000000&0.000000&0.000000&0.000000&0.000000&0.000000&0.000000\\\hline
\end{tabular}
$\vspace{0.03cm}$
    \caption{Probability distribution at arbitrary slot for single vacation }\lb{tb2}$\vspace{0.01cm}$
\begin{tabular}{|c|c|c|c|c|c|c|c|c|c|c|c|c|} \hline
$n$& $p_{n,0}$ & $p_{n,8}$ & $p_{n,9}$& $p_{n,10}$& $p_{n,11}$& $p_{n,12}$&$p_{n,13}$&$p_{n,14}$&$p_{n,15}$&$Q_n$&$p_n^{queue}$ \\\hline
0&0.112501&0.034574&0.037555&0.031608&0.018848&0.001786&0.001479&0.001277&0.001087&0.016369&0.257089\\
1&0.046329&0.002480&0.004440&0.004630&0.003113&0.000318&0.000278&0.000249&0.001124&0.002410&0.065375\\
2&0.060140&0.002645&0.004943&0.005289&0.003617&0.000374&0.000329&0.000298&0.001208&0.002699&0.081546\\
3&0.077683&0.002821&0.005500&0.006039&0.004200&0.000440&0.000391&0.000355&0.001341&0.003025&0.101799\\
4&0.099924&0.003008&0.006119&0.006893&0.004876&0.000517&0.000463&0.000424&0.001518&0.003391&0.127137\\
5&0.081922&0.000726&0.002364&0.003234&0.002546&0.000288&0.000271&0.000256&0.001527&0.001418&0.094557\\
6&0.089928&0.000606&0.002118&0.003022&0.002444&0.000282&0.000269&0.000257&0.001544&0.001303&0.101777\\
7&0.096110&0.000468&0.001790&0.002687&0.002246&0.000265&0.000256&0.000248&0.001559&0.001139&0.106773\\
10&&0.000093&0.000653&0.001286&0.001259&0.000165&0.000171&0.000175&0.001495&0.000050&0.005352\\
25&&0.000000&0.000004&0.000031&0.000067&0.000014&0.000022&0.000030&0.000751&0.000000&0.000922\\
50&&0.000000&0.000000&0.000000&0.000000&0.000000&0.000000&0.000012&0.000155&0.000000&0.000157\\
75&&0.000000&0.000000&0.000000&0.000000&0.000000&0.000000&0.000002&0.000029&0.000000&0.000029\\
100&&0.000000&0.000000&0.000000&0.000000&0.000000&0.000000&0.000000&0.000005&0.000000&0.000005\\
$\geq$120&&0.000000&0.000000&0.000000&0.000000&0.000000&0.000000&0.000000&0.000000&0.000000&0.000000\\\hline
\multicolumn{12}{|c|}{$L$=7.239171,~~ $L_q$=4.052410,~~ $L_{s}$=10.509858,}\\
\multicolumn{12}{|c|}{$P_{busy}$=0.303392,~~~ $W$=52.648521,~~ $W_q$=29.472075}\\\hline
\end{tabular}
\end{tiny}
\end{table}

\begin{table}
\centering
\begin{tiny}
    \caption{ Probability distribution at service completion epoch for multiple vacation}\lb{tb3}$\vspace{0.03cm}$

\begin{tabular}{|c|c|c|c|c|c|c|c|c|c|c|} \hline
$n$& $p_{n,8}^+$& $p_{n,9}^+$& $p_{n,10}^+$& $p_{n,11}^+$& $p_{n,12}^+$& $p_{n,13}^+$&$p_{n,14}^+$&$p_{n,15}^+$& $p_n^+$&$Q_n^+$\\\hline
0&0.011188&0.006086&0.003491&0.001684&0.000339&0.000207&0.000127&0.000076&0.023201&0.159460\\
1&0.001057&0.000872&0.000599&0.000319&0.000068&0.000043&0.000027&0.000069&0.003057&0.063198\\
2&0.001135&0.000979&0.000690&0.000374&0.000081&0.000052&0.000033&0.000075&0.003421&0.081806\\
3&0.001218&0.001098&0.000795&0.000438&0.000096&0.000062&0.000040&0.000085&0.003834&0.105439\\
4&0.001307&0.001231&0.000915&0.000513&0.000113&0.000074&0.000048&0.000097&0.004302&0.135398\\
5&0.000345&0.000508&0.000453&0.000280&0.000066&0.000045&0.000030&0.000097&0.001827&0.110483\\
10&0.000048&0.000153&0.000196&0.000150&0.000040&0.000030&0.000022&0.000100&0.000743&0.008535\\
25&0.000000&0.000001&0.000006&0.000009&0.000004&0.000004&0.000004&0.000056&0.000086&0.000000\\
50&0.000000&0.000000&0.000000&0.000000&0.000000&0.000000&0.000000&0.000012&0.000012&0.000000\\
75&0.000000&0.000000&0.000000&0.000000&0.000000&0.000000&0.000000&0.000002&0.000002&0.000000\\
100&0.000000&0.000000&0.000000&0.000000&0.000000&0.000000&0.000000&0.000000&0.000000&0.000000\\
$\geq$120&0.000000&0.000000&0.000000&0.000000&0.000000&0.000000&0.000000&0.000000&0.000000&0.000000\\\hline
\end{tabular}
\end{tiny}
\end{table}
\newpage
\begin{table}
\centering
\begin{tiny}
    \caption{Probability distribution at arbitrary slot for multiple vacation }\lb{tb4}$\vspace{0.03cm}$
\begin{tabular}{|c|c|c|c|c|c|c|c|c|c|c|c|c|} \hline
$n$& $p_{n,0}$ & $p_{n,8}$ & $p_{n,9}$& $p_{n,10}$& $p_{n,11}$& $p_{n,12}$&$p_{n,13}$&$p_{n,14}$&$p_{n,15}$&$Q_n$&$p_n^{queue}$ \\\hline
0&0.073324&0.002728&0.003043&0.002661&0.001731&0.000439&0.000324&0.000234&0.000160&0.010668&0.095316\\
1&0.047391&0.000195&0.000359&0.000389&0.000285&0.000078&0.000061&0.000045&0.000142&0.004073&0.053024\\
2&0.067796&0.000208&0.000400&0.000445&0.000332&0.000092&0.000072&0.000054&0.000152&0.005258&0.074813\\
3&0.095612&0.000222&0.000445&0.000508&0.000385&0.000108&0.000085&0.000065&0.000168&0.006763&0.104366\\
4&0.133291&0.000237&0.000495&0.000580&0.000447&0.000127&0.000101&0.000077&0.000191&0.008669&0.144220\\
5&0.136826&0.000057&0.000191&0.000272&0.000233&0.000071&0.000059&0.000047&0.000188&0.007031&0.144979\\
6&0.164142&0.000047&0.000171&0.000254&0.000224&0.000069&0.000059&0.000047&0.000190&0.007709&0.172916\\
7&0.192068&0.000036&0.000145&0.000226&0.000206&0.000065&0.000056&0.000045&0.000192&0.008232&0.201275\\
25&&0.000000&0.000000&0.000002&0.000006&0.000004&0.000005&0.000006&0.000094&0.000000&0.000118\\
50&&0.000000&0.000000&0.000000&0.000000&0.000000&0.000000&0.000000&0.000019&0.000000&0.000019\\
75&&0.000000&0.000000&0.000000&0.000000&0.000000&0.000000&0.000000&0.000004&0.000000&0.000004\\
100&&0.000000&0.000000&0.000000&0.000000&0.000000&0.000000&0.000000&0.000000&0.000000&0.000000\\
$\geq$120&&0.000000&0.000000&0.000000&0.000000&0.000000&0.000000&0.000000&0.000000&0.000000&0.000000\\\hline
\multicolumn{12}{|c|}{$L$=4.746889,~~ $L_q$=4.422526,~~ $L_{s}$=10.999253,}\\
\multicolumn{12}{|c|}{$P_{busy}$=0.029511,~~~ $W$=34.522829,~~ $W_q$=32.163831}\\\hline
\end{tabular}
\end{tiny}
\end{table}

\newpage
\textbf{Example 2:} In this example, the service time has been considered as negative binomial (NB) distribution with pmf
\bea &&s_i(n)=\binom{n-1}{r-1}\mu_i^r(1-\mu_i)^{n-r},~n=r,r+1,r+2,\ldots ~ \mbox{which consequently leads}\nn\\&&~ K^{(i)}(x)=\left\{\frac{\mu_i(\bar\la+\la G(x))}{1-(1-\mu_i)(\bar\la+\la G(x))}\right\}^r.\nn\eea
In particular, here $r=3,~\la=0.10,~a=15,~b=25,~g_1=0.5,~g_5=0.5, \bar g=3.0$ and service rates are taken as follows: $\mu_{15}=0.300000$, $\mu_{16}=0.250000$, $\mu_{17}=0.214285$, $\mu_{18}=0.187500$, $\mu_{19}=0.166667$, $\mu_{20}=0.150000$, $\mu_{21}=0.136364$, $\mu_{22}=0.125000$, $\mu_{23}=0.115385$, $\mu_{24}=0.107142$, $\mu_{25}=0.100000$ so that $\rho=0.360000$. The vacation time follows geometric distribution with pmf $(v_j)=\eta~(1-\eta)^{j-1},~~j\geq 1$ with $\eta=0.2$, pgf $=\dfrac{\eta z}{1-(1-\eta)z}$, mean ($\bar v$)=5.0. The joint distribution at service completion and arbitrary epochs along with some performance measures are displayed in Tables [\ref {tb5}-\ref {tb8}].

\begin{sidewaystable}
\begin{tiny}
\caption{Probability distribution at service completion epoch for single vacation}\lb{tb5}$\vspace{0.03cm}$
\begin{tabular}{|c|c|c|c|c|c|c|c|c|c|c|c|c|c|} \hline
$n$& $p_{n,15}^+$& $p_{n,16}^+$& $p_{n,17}^+$& $p_{n,18}^+$& $p_{n,19}^+$& $p_{n,20}^+$&$p_{n,21}^+$&$p_{n,22}^+$&$p_{n,23}^+$&$p_{n,24}^+$&$p_{n,25}^+$& $p_n^+$&$Q_n^+$\\\hline
0&0.055782&0.035631&0.029595&0.020952&0.013528&0.000445&0.000387&0.000211&0.000089&0.000033&0.000051&0.156709&0.100741\\
1&0.025127&0.018272&0.016843&0.012993&0.009019&0.000316&0.000290&0.000165&0.000073&0.000028&0.000098&0.083228&0.073492\\
2&0.006149&0.005231&0.005454&0.004650&0.003507&0.000131&0.000128&0.000077&0.000035&0.000014&0.000107&0.025490&0.029857\\
3&0.001127&0.001142&0.001362&0.001294&0.001067&0.000043&0.000045&0.000028&0.000013&0.000005&0.000087&0.006216&0.009172\\
4&0.000176&0.000214&0.000294&0.000312&0.000283&0.000012&0.000013&0.000009&0.000004&0.000002&0.000058&0.001381&0.002420\\
5&0.025151&0.018309&0.016901&0.013062&0.009087&0.000319&0.000294&0.000168&0.000074&0.000029&0.000086&0.083485&0.074052\\
10&0.006169&0.005267&0.005519&0.004736&0.003602&0.000136&0.000134&0.000082&0.000038&0.000015&0.000085&0.025788&0.030602\\
20&0.000178&0.000221&0.000312&0.000344&0.000324&0.000014&0.000017&0.000012&0.000006&0.000003&0.000041&0.001478&0.000872\\
40&0.000001&0.000001&0.000000&0.000000&0.000000&0.000000&0.000000&0.000000&0.000000&0.000000&0.000003&0.000005&0.000000\\
$\geq$50&0.000000&0.000000&0.000000&0.000000&0.000000&0.000000&0.000000&0.000000&0.000000&0.000000&0.000000&0.000000&0.000000\\\hline
\end{tabular}
$\vspace{0.03cm}$
    \caption{Probability distribution at arbitrary slot for single vacation }\lb{tb6}$\vspace{0.01cm}$
\begin{tabular}{|c|c|c|c|c|c|c|c|c|c|c|c|c|c|c|c|} \hline
$n$& $p_{n,0}$ & $p_{n,15}$ & $p_{n,16}$& $p_{n,17}$& $p_{n,18}$& $p_{n,19}$&$p_{n,20}$&$p_{n,21}$&$p_{n,22}$&$p_{n,23}$&$p_{n,24}$&$p_{n,25}$&$Q_n$&$p_n^{queue}$ \\\hline
0&0.035348&0.030797&0.025176&0.025977&0.022344&0.017229&0.000668&0.000677&0.000425&0.000206&0.000087&0.000151&0.019638&0.178730\\
1&0.043461&0.006581&0.006176&0.007078&0.006613&0.005450&0.000223&0.000236&0.000154&0.000077&0.000033&0.000218&0.013235&0.089542\\
2&0.032207&0.001133&0.001252&0.001625&0.001674&0.001494&0.000065&0.000073&0.000050&0.000026&0.000011&0.000193&0.005085&0.044893\\
3&0.019322&0.000170&0.000225&0.000334&0.000383&0.000372&0.000017&0.000020&0.000015&0.000008&0.000002&0.000131&0.001505&0.022512\\
4&0.010510&0.000023&0.000037&0.000064&0.000081&0.000087&0.000004&0.000005&0.000004&0.000002&0.000001&0.000076&0.000388&0.011287\\
5&0.048913&0.006584&0.006182&0.007090&0.006629&0.005469&0.000224&0.000238&0.000155&0.000078&0.000034&0.000121&0.013322&0.095046\\
10&0.049000&0.001135&0.001257&0.001637&0.001694&0.001519&0.000066&0.000075&0.000051&0.000027&0.000012&0.000079&0.00519&0.061759\\
20&&0.000023&0.000038&0.000067&0.000088&0.000097&0.000005&0.000006&0.000005&0.000003&0.000001&0.000026&0.000122&0.000486\\
40&&0.000000&0.000000&0.000000&0.000000&0.000000&0.000000&0.000000&0.000000&0.000000&0.000000&0.000001&0.000000&0.000001\\
$\geq$50&&0.000000&0.000000&0.000000&0.000000&0.000000&0.000000&0.000000&0.000000&0.000000&0.000000&0.000000&0.000000&0.000000\\\hline
\multicolumn{15}{|c|}{$L$=10.416824,~~ $L_q$=6.291121,~~ $L_{s}$=17.093524,}\\
\multicolumn{15}{|c|}{$P_{busy}$=0.241367,~~~ $W$=34.722748,~~ $W_q$=20.970406}\\\hline
\end{tabular}
\end{tiny}
\end{sidewaystable}

\begin{sidewaystable}
\begin{tiny}
\caption{Probability distribution at service completion epoch for multiple vacation}\lb{tb7}$\vspace{0.03cm}$
\begin{tabular}{|c|c|c|c|c|c|c|c|c|c|c|c|c|c|} \hline
$n$& $p_{n,15}^+$& $p_{n,16}^+$& $p_{n,17}^+$& $p_{n,18}^+$& $p_{n,19}^+$& $p_{n,20}^+$&$p_{n,21}^+$&$p_{n,22}^+$&$p_{n,23}^+$&$p_{n,24}^+$&$p_{n,25}^+$& $p_n^+$&$Q_n^+$\\\hline
0&0.008151&0.006092&0.004949&0.003601&0.002400&0.000907&0.000585&0.000462&0.000345&0.000238&0.000110&0.027845&0.050122\\
1&0.003671&0.003124&0.002817&0.002233&0.001600&0.000643&0.000438&0.000362&0.000282&0.000202&0.000172&0.015549&0.055834\\
2&0.000898&0.000894&0.000912&0.000799&0.000622&0.000268&0.000194&0.000169&0.000138&0.000103&0.000181&0.005182&0.038801\\
3&0.000164&0.000195&0.000227&0.000222&0.000189&0.000088&0.000067&0.000062&0.000053&0.000041&0.000165&0.001479&0.022582\\
4&0.000025&0.000036&0.000049&0.000053&0.000050&0.000025&0.000020&0.000020&0.000018&0.000014&0.000139&0.000455&0.012258\\
5&0.003675&0.003130&0.002826&0.002245&0.001612&0.000649&0.000444&0.000368&0.000288&0.000207&0.000202&0.015651&0.062193\\
10&0.000901&0.000900&0.000923&0.000814&0.000639&0.000278&0.000203&0.000179&0.000148&0.000112&0.000198&0.005299&0.058150\\
20&0.000026&0.000037&0.000052&0.000059&0.000057&0.000030&0.000026&0.000027&0.000025&0.000022&0.000088&0.000453&0.004331\\
40&0.000000&0.000000&0.000000&0.000000&0.000000&0.000000&0.000000&0.000000&0.000000&0.000000&0.000005&0.000005&0.000006\\
$\geq$50&0.000000&0.000000&0.000000&0.000000&0.000000&0.000000&0.000000&0.000000&0.000000&0.000000&0.000000&0.000000&0.000000\\\hline
\end{tabular}
$\vspace{0.03cm}$
    \caption{Probability distribution at arbitrary slot for multiple vacation }\lb{tb8}$\vspace{0.01cm}$
\begin{tabular}{|c|c|c|c|c|c|c|c|c|c|c|c|c|c|c|c|} \hline
$n$& $p_{n,0}$ & $p_{n,15}$ & $p_{n,16}$& $p_{n,17}$& $p_{n,18}$& $p_{n,19}$&$p_{n,20}$&$p_{n,21}$&$p_{n,22}$&$p_{n,23}$&$p_{n,24}$&$p_{n,25}$&$Q_n$&$p_n^{queue}$ \\\hline
0&0.016404&0.004197&0.004015&0.004052&0.003582&0.002851&0.001269&0.000954&0.000868&0.000742&0.0005821&0.000303&0.009113&0.048937\\
1&0.026475&0.000897&0.000985&0.001104&0.001060&0.000902&0.000424&0.000333&0.000315&0.000278&0.000224&0.000329&0.009645&0.042976\\
2&0.025936&0.000154&0.000199&0.000253&0.000268&0.000247&0.000124&0.000103&0.000102&0.000094&0.000078&0.000301&0.006519&0.034383\\
3&0.020359&0.000023&0.000035&0.000052&0.000061&0.000061&0.000033&0.000029&0.000030&0.000029&0.000025&0.000259&0.003743&0.024745\\
4&0.014191&0.000003&0.000005&0.000009&0.000013&0.000014&0.000008&0.000008&0.000008&0.000008&0.000008&0.000210&0.002020&0.016510\\
5&0.035652&0.000897&0.000986&0.001106&0.001063&0.000905&0.000426&0.000335&0.000317&0.000281&0.000227&0.000265&0.010689&0.053153\\
10&0.060181&0.000154&0.000200&0.000255&0.000271&0.000251&0.000126&0.000106&0.000106&0.000098&0.000082&0.000174&0.009689&0.071698\\
20&&0.000003&0.000006&0.000010&0.000014&0.000016&0.000009&0.000009&0.000010&0.000011&0.000010&0.000052&0.000566&0.000722\\
40&&0.000000&0.000000&0.000000&0.000000&0.000000&0.000000&0.000000&0.000000&0.000000&0.000000&0.000002&0.000000&0.000002\\
$\geq$50&&0.000000&0.000000&0.000000&0.000000&0.000000&0.000000&0.000000&0.000000&0.000000&0.000000&0.000000&0.000000&0.000000\\\hline
\multicolumn{15}{|c|}{$L$=9.752527,~~ $L_q$=8.769886,~~ $L_{s}$=18.728021,}\\
\multicolumn{15}{|c|}{$P_{busy}$=0.052468,~~~ $W$=32.508423,~~ $W_q$=29.232956}\\\hline
\end{tabular}
\end{tiny}
\end{sidewaystable}

\newpage
\section{Conclusion}\label{CN}
In this paper, we have successfully analyzed a discrete-time group-arrival and batch-service queue with batch-size-dependent service policy along with single and multiple vacation. Employing generating function technique, we have derived the bivariate probability generating function of queue length and server content distribution at service completion epoch. The joint probabilities are extracted and presented in a simple and tractable form. We have also acquired
the joint distribution at arbitrary slot after establishing a relationship with the probabilities at service completion epoch. Several noteworthy marginal distribution and performance measures are reported. The analytical results have been illustrated numerically with significant distributions which can cope up with real-life situations. At the end we hope that the results will be fruitful to the vendors/system designers from application perspective.

\bibliographystyle{unsrtnat}
\bibliography{discretevacation}

\end{document}